\documentclass[a4paper,12pt]{article}
\usepackage[centertags]{amsmath}
\usepackage{amsfonts}
\usepackage{amssymb}
\usepackage{amsthm}
\usepackage{amsmath}
\usepackage{dsfont}
\usepackage{graphicx} 
\usepackage{tikz, subfigure}
\usepackage{pstricks-add}
\usepackage{verbatim}

\usepackage[latin1]{inputenc}
\usepackage{tikz}
\definecolor {processblue}{cmyk}{0.96,0,0,0}
\usepackage{amsfonts,graphicx,amsmath,amssymb,hyperref,color}
\usepackage[english]{babel}
\usepackage{authblk}

\newtheorem{Theorem}{Theorem}[section]

\newtheorem{Lemma}{Lemma}[section]
\newtheorem{Corollary}{Corollary}[section]
\newtheorem{Conjecture}{Conjecture}[section]
\newtheorem{Definition}{Definition}[section]

\newtheorem{Example}{Example}[section]

\newtheorem{Remark}{Remark}[section]

\newtheorem{Problem}{Problem}[section]

\AtEndDocument{\bigskip{\footnotesize%
  \textsc{Department of Mathematics, Utrecht University, Fac Wiskunde en informatica and MRI, Budapestlaan 6, P.O. Box 80.000, 3508 TA Utrecht, The Netherlands} \par
  \textit{E-mail address:} \texttt{k.dajani1@uu.nl} \par
}}
\AtEndDocument{\bigskip{\footnotesize%
  \textsc{Department of Mathematics, Ningbo University, Ningbo, Zhejiang, People's Republic of China} \par
  \textit{E-mail address:} \texttt{kanjiangbunnik@yahoo.com} \par
}}

\usepackage{lipsum}

\makeatletter
\newcommand*{\rom}[1]{\expandafter\@slowromancap\romannumeral #1@}
\makeatother
\begin{document}
\title{xxxx}
\date{}
 \title{On the points without universal expansions}
 \author{Karma Dajani and Kan Jiang\thanks{   Corresponding author.} }
\maketitle{}
\begin{abstract}
Let  $1<\beta<2$. Given any $x\in[0, (\beta-1)^{-1}]$, a sequence $(a_n)\in\{0,1\}^{\mathbb{N}}$ is called a  $\beta$-expansion  of $x$ if 
$x=\sum_{n=1}^{\infty}a_n\beta^{-n}.$
For any $k\geq 1$  and any $(b_1b_2\cdots b_k)\in\{0,1\}^{k}$, if there exists some $k_0$ such that 
$a_{k_0+1}a_{k_0+2}\cdots a_{k_0+k}=b_1b_2\cdots b_k$, then we call $(a_n)$ a universal $\beta$-expansion of $x$. 
  Sidorov \cite{Sidorov2003}, Dajani and  de Vries   \cite{DajaniDeVrie} proved that given any $1<\beta<2$, then Lebesgue  almost every point has uncountably many universal expansions. In this paper we consider the set $V_{\beta}$ of points without universal expansions. For any $n\geq 2$, let $\beta_n$ be the  $n$-bonacci number satisfying the following equation: $\beta^n=\beta^{n-1}+\beta^{n-2}+\cdots +\beta+1.$
Then we have 
$\dim_{H}(V_{\beta_n})=1$, where $\dim_{H}$ denotes the Hausdorff dimension. Similar results  are still available for some other algebraic numbers. As a corollary, we give some results of  the Hausdorff  dimension of the survivor set generated by some open dynamical systems.  This note is another application of our paper \cite{KarmaKan}. 
\end{abstract}

\section{Introduction}
Let  $1<\beta<2$. Given any $x\in[0, (\beta-1)^{-1}]$, a sequence $(a_n)\in\{0,1\}^{\mathbb{N}}$ is called a  $\beta$-expansion  of $x$ if 
$$x=\sum\limits_{n=1}^{\infty}\dfrac{a_n}{\beta^n}.$$
Sidorov \cite{Sidorov} proved that given any $1<\beta<2$, then almost every point in $[0, (\beta-1)^{-1}]$ has uncountably many expansions.  If $(a_n)$ is the only $\beta$-expansion of $x$, then we call $x$ a univoque point with unique expansion $(a_n)$. Denote by $U_{\beta}$ all the univoque points in base $\beta.$ For the unique expansions, there are many results, see \cite{MK, GS} and references therein. 

Let $(a_n)$ be a $\beta$-expansion of $x$. If for any $k\geq 1$, and any $(b_1b_2\cdots b_k)\in\{0,1\}^{k}$ there exists some $k_0$ such that 
$$a_{k_0+1}a_{k_0+2}\cdots a_{k_0+k}=b_1b_2\cdots b_k,$$ then we call $(a_n)$ a universal $\beta$-expansion of $x$.

The  dynamical approach is a good way which can generate  $\beta$-expansions  effectively. Define $T_0(x)=\beta x, T_1(x)=\beta x-1$, see Figure 1. 

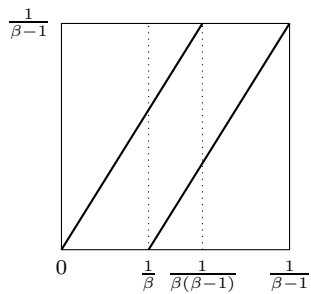
\begin{figure}[h]\label{figure1}
\centering
\begin{tikzpicture}[scale=3]
\draw(0,0)node[below]{\scriptsize 0}--(.382,0)node[below]{\scriptsize$\frac{1}{\beta}$}--(.618,0)node[below]{\scriptsize$\frac{1}{\beta(\beta-1)}$}--(1,0)node[below]{\scriptsize$\frac{1}{\beta-1}$}--(1,1)--(0,1)node[left]{\scriptsize$\frac{1}{\beta-1}$}--(0,.5)--(0,0);
\draw[dotted](.382,0)--(.382,1)(0.618,0)--(0.618,1);
\draw[thick](0,0)--(0.618,1)(.382,0)--(1,1);
\end{tikzpicture}\caption{The dynamical system for $\{T_{0},\,T_{1}\}$}
\end{figure}

Let $x\in [0, (\beta-1)^{-1}]$ with an expansion $(a_n)_{n=1}^{\infty}$, and set $T_{a_1a_2\cdots a_n}=T_{a_n}\circ T_{a_{n-1}}\circ \cdots \circ T_{a_1}$. We call $\{T_{a_1a_2\cdots a_n}(x)\}_{n=0}^{\infty}$ an orbit of $x$ in base $\beta$. For simplicity, we  denote by $T_{a_0}(x)=x$.  Clearly, for different expansions, $x$ has distinct orbits. 
Evidently,  for any $n\geq 1$, we have
$$x=\dfrac{a_1}{\beta}+\dfrac{a_2}{\beta^2}+\cdots+\dfrac{a_n}{\beta^n}+\dfrac{T_{a_1a_2\cdots a_n}(x)}{\beta^n}.$$
The digits $(a_n)$ are chosen in the following way:
if $T_{a_1a_2\cdots a_{j-1}}(x)\in [0,\beta^{-1})$, then  $a_j=0$, if  $T_{a_1a_2\cdots a_{j-1}}(x)\in((\beta-1)^{-1}\beta^{-1},(\beta-1)^{-1}]$, then $a_j=1$. 
However, 
if $T_{a_1a_2\cdots a_{j-1}}(x)\in [\beta^{-1}, (\beta-1)^{-1}\beta^{-1}]$, then we may choose $a_j$  to be $0$ or $1$. Due to this observation, we call $ [\beta^{-1}, (\beta-1)^{-1}\beta^{-1}]$ the switch region. All  the possible  $\beta$-expansions can be generated in terms of this idea, see \cite{KarmaCor, KM}. If $x$ has exactly $k$ different expansions, then we say $x$ has multiple expansions \cite{SN,DJKL,DJKL1}. 

For the set of unique expansions, one has criteria that characterizes this type of  expansions \cite{MK,GS}.  However, for the universal expansions and multiple expansions,  few papers considered this aspect. In this paper, we shall use the dynamical approach to study the universal expansions. 

Universal expansions have a close connection with the following discrete spectra
$$D=\left\{\sum_{i=0}^{n}a_i\beta^i:a_i\in\{0,1\}, n\geq 0\right\}.$$
Denote by $D=\{y_0=0<y_1<y_2<\cdots<y_k<\cdots\}$. 
Define $$L^{1}(\beta)=\limsup_{k\to\infty}(y_{k+1}-y_k).$$
Erd\"{o}s and Komornik \cite{EK} proved that if $L^{1}(\beta)=0$, then all the points of $(0, (\beta-1)^{-1})$ have universal expansions. Moreover, Erd\"{o}s and  Komornik \cite{EK}
 showed that if $1<\beta\leq\sqrt[4]{2}\approx1.19$, then $L^{1}(\beta)=0$.  In particular,  Erd\"{o}s and  Komornik \cite{EK} proved that $L^{1}(\sqrt{2})=0$. 
Sidorov and Solomyak \cite{SidorovSolomyak} also considered some algebraic numbers for which $L^{1}(\beta)=0$, and their result  is improved by Komornik and Akiyama \cite{AK}. In \cite{AK},  Akiyama  and Komornik proved that if  $1<\beta\leq\sqrt[3]{2}\approx 1.26$, then $L^{1}(\beta)=0$.  In \cite{Dejun},  Feng utilized Akiyama  and Komornik's result \cite{AK}, and implememted some ideas in fractal geometry showing that for any  non-Pisot  $\beta\in(1, \sqrt{2}]$ if $\beta^2$ is not a Pisot number, then $L^{1}(\beta)=0$. For the generic results,  Sidorov \cite{Sidorov2003} showed that given any $1<\beta<2$, almost every point in $[0, (\beta-1)^{-1}]$ has at least one universal expansion. Dajani and de Vries \cite{DajaniDeVrie}, used a dynamical approach to show that for any $\beta>1$, almost every point of $[0, (\beta-1)^{-1}]$ has uncountably many universal expansions.

The results of Sidorov\cite{Sidorov} and those of Dajani and de Vries \cite{DajaniDeVrie} imply that the set of points without universal expansions has zero Lebesgue measure.  In other words, the Lebesgue measure of  $V_{\beta}$ is zero, where 
$$V_{\beta}=\{x\in [0, (\beta-1)^{-1}]: x \textrm{ does not have a universal expansion}\}.$$
A natural question is to study the Hausdorff dimension of the set $V_{\beta}$. This is  the main motivation of this paper.  
 Using one property of Pisot numbers,  we have the following result. 
\begin{Theorem}\label{MainTheorem}
For any $n\geq 2$, let $\beta_n$ be the  n-bonacci number satisfying the following equation:
 $$\beta^n=\beta^{n-1}+\beta^{n-2}+\cdots +\beta+1,$$ then $\dim_{H}(V_{\beta_n})=1$.
\end{Theorem}
For some Pisot numbers, we have similar results. 
 For $1<\beta<\dfrac{1+\sqrt{5}}{2}$,
interestingly, 
the Hausdorff dimension of $V_{\beta}$ has a close connection with an old conjecture posed by
Erd\"{o}s and Komornik \cite{EK}. 
\setcounter{Conjecture}{1}
\begin{Conjecture}
For any non-Pisot $\beta\in\left(1,\dfrac{1+\sqrt{5}}{2}\right)$, $L^{1}(\beta)=0$.
\end{Conjecture}
 This conjecture is true if $\beta\in(1,\sqrt{2}]$ and $\beta^2$ is not a Pisot number, see \cite{Dejun}. We
make a brief discussion of the connection between the dimension of $V_\beta$ and this conjecture. 
If we were able to find some non-Pisot number $1<\beta<\dfrac{1+\sqrt{5}}{2}$ such that the Hausdorff dimension of $V_{\beta}$ is positive, then $L^{1}(\beta)>0$. The reason is due  to the fact that $L^{1}(\beta)=0$ implies all the points of $(0, (\beta-1)^{-1})$ have universal expansions. In other words, we disprove the Erd\"{o}s-Komornik conjecture. 
Therefore, considering the Hausdorff dimension of $V_{\beta}$ is meaningful to this conjecture. 
The dimensional problem of $V_{\beta}$ has a strong relation with open dynamical systems. Roughly speaking, 
$V_{\beta}$ is a union of  countable survivor sets generated by some open dynamical systems. These open dynamical systems are smaller than the ususal open systems as we consider all the possible orbits, i.e. all the possible orbits should avoid some holes. In this paper, we shall make use of this tool to study the dimension of $V_{\beta}$. 

The paper is arranged as follows. In section 2,  we start with some necessary  definitions and notation, then we state the main results of the paper. In section 3, we give the proofs, and in section 4 
we give some final remarks.
\section{ Preliminaries and Main results}
In this section, we give some notation and definitions. 
Let 
 $\Omega=\{0,1\}^{\mathbb{N}}$,  $E=[0,(\beta-1)^{-1}]$, and $\sigma$ be the left shift. The  random $\beta$-transformation $K$  is defined in the following way, see \cite{KarmaCor}.
 \begin{Definition}
 $K:\Omega\times E\to \Omega\times E$ is defined by
\begin{equation*}
K(\omega, x)=\left\lbrace\begin{array}{cc}
                (\omega, \beta x)& x\in[0,\beta^{-1})\\
       (\sigma(\omega), \beta x-\omega_1)& x\in[\beta^{-1},\beta^{-1}(\beta-1)^{-1}]\\
              (\omega, \beta x-1)& x\in(\beta^{-1}(\beta-1)^{-1},(\beta-1)^{-1}]\\
                \end{array}\right.
\end{equation*}
\end{Definition}
We call $[\beta^{-1},\beta^{-1}(\beta-1)^{-1}]$ the switch region since in this region we can choose the  digit to be usedand change from 0 to 1 or vica versa..

When the orbits of points hit or enter the switch region and we always choose the digit $1$, then we call this algorithm the greedy algorithm.  More precisely, the greedy map is defined in the following way:
 $G: E\to  E$ is defined by
\begin{equation*}
G(x)=\left\lbrace\begin{array}{cc}
             \beta x& x\in[0,\beta^{-1})\\
          \beta x-1& x\in[\beta^{-1},(\beta-1)^{-1}]\\
                \end{array}\right.
\end{equation*}
Let $(\omega, x)\in \Omega\times E$. For any $n\geq 1$,  we denote by $K^{n}(\omega, x)=K(K^{n-1}(\omega, x)) $ the $n$ iteration of $K$, and  let  $\pi(\omega, x))=x$ be the projection in the second  coordinate
We can study $\beta$-expansions via the following iterated function system,
\[f_j(x)=\dfrac{x+j}{\beta},\,j\in\{0,1\}.\]
The self-similar set \cite{Hutchinson} for this IFS is the interval
$\left[0,(\beta-1)^{-1}\right]$. This tool is  useful in the proof of Lemma \ref{lem:315}. 
Before we state our main results, we   define some sets. 
Given  $1<\beta<2$ and any $N\geq 3$. Define $$E_{\beta, N}=\{x\in[0, (\beta-1)^{-1}]: \textrm{ no orbit of } x \textrm{ hits }[0,\beta^{-N}(\beta-1)^{-1}]\},$$
$$F_{\beta, N}=\{x\in[0, (\beta-1)^{-1}]: \textrm{ the greedy orbit of  } x \textrm{ does not hit }[0,\beta^{-N}(\beta-1)^{-1}]\}.$$
We can give a simple symbolic explanation of $E_{\beta, N}$, namely, any $\beta$-expansion $(a_n)$ of any point in $E_{\beta, N}$ does not contain the block $\underbrace{(00\cdots 0)}_{N}$.
Let $$\mathcal{O}=\{\pi(K^{n}(\omega, 1))\cup \pi(K^{n}(\omega, (\beta-1)^{-1}-1):n\geq 0, \omega\in \Omega\}$$
 be the set of all possible orbits of $1$ and $(\beta-1)^{-1}-1.$
 An algebraic number $\beta>1$ is called a Pisot number if all of its conjugates lie inside the unit circle. 
 Now we state our main results of this paper. 
 \setcounter{Theorem}{1}
\begin{Theorem}\label{thm:21}
For any $n\geq 2$, let $\beta_n$ be the  n-bonacci number satisfying the following equation:
 $$\beta^n=\beta^{n-1}+\beta^{n-2}+\cdots +\beta+1,$$ then $\dim_{H}(V_{\beta_n})=1$.
\end{Theorem} 
The followig result gives a sufficient condition under which the Hausdorff dimension of $F_{\beta, N}$ can be calculated.

\setcounter{Corollary}{2}
\begin{Corollary}\label{cor:23}
Let   $\dfrac{1+\sqrt{5}}{2}<\beta<2$. If all the possible orbits of $1$ hit finite points, then given any $N\geq 3$, $\dim_{H}(F_{\beta, N})$ can be calculated explicitly. In particularly, for any Pisot number in $(1,2)$, $\dim_{H}(F_{\beta, N})$ can be calculated.
\end{Corollary}
This result is indeed a corollary of \cite[Theorem 4.2]{KarmaKan}. 
Generally,  calculating the Hausdorff dimension of $E_{\beta, N}$ is not an easy problem. By definition, $E_{\beta, N}\subset F_{\beta, N}$ for any $N\geq 3.$ Hence, $E_{\beta, N}$ is a smaller survivor set, and it is difficult to  calculate the dimension of this set. However, for the sequence $(\beta_n)$, we have the following asymptotic result. 
\setcounter{Theorem}{3}
\begin{Theorem}\label{thm:24}
For any $n\geq 2$ and $N\geq 3$, let $\beta_n$ be the  n-bonacci number, then $\dim_{H}(F_{\beta_n, N-1})\leq \dim_{H}(E_{\beta_n, N})\leq \dim_{H}(F_{\beta_n, N})$. Subsequently, 
  $$\lim_{N\to \infty}\dim_{H}(E_{\beta_n, N})=\lim_{N\to \infty}\dim_{H}(F_{\beta_n, N})=1.$$
  Moreover, for any $N>2n+4$, $\dim_{H}(F_{\beta_n, N}\setminus E_{\beta_n, N})>0.$
Furthermore, we can find some set with positive Hausdorff dimension such that every point in this set has uncountably many expansions, but none of them is a universal expansion. 
\end{Theorem}
The last statement strengthens  one result of \cite[Counterexample]{Sidorov2003}. 
The following result is about the topological structure of $E_{\beta, N}$. 
\begin{Theorem}\label{thm:25}
Given any $N\geq 3$,
for almost every $\beta\in(1,2)$,   $E_{\beta, N}$ is a graph-directed self-similar set. 
\end{Theorem}
\section{Proof of  main theorems}
In this section, we  give a proof of  Theorem \ref{MainTheorem}.
To begin with, we recall some classical results and notation. 
An expansion $(a_n)$ is called the quasi-greedy expansion if it is the largest infinite expansions, in the  sense of lexicographical ordering. Denote by $\sigma((a_n)_{n=1}^{\infty})=(a_n)_{n=2}^{\infty}$, and $\sigma^k((a_n)_{n=1}^{\infty})=(a_n)_{n=k+1}^{\infty}$.  
Let $(\alpha_n)$ be the quasi-greedy expansion of $1$.  The following classical result was proved by Parry \cite{Parry}. 
\begin{Theorem}\label{Uniquecodings}
 Let $(a_n)_{n=1}^{\infty}$ be an expansion of $x\in [0,(\beta-1)^{-1}]$. Then $(a_n)_{n=1}^{\infty}$ is a greedy expansion if and only if \[\sigma^{k}((a_n)_{n=1}^{\infty})<(\alpha_n)_{n=1}^{\infty}\] if $a_k=0$.
\end{Theorem}
\setcounter{Lemma}{1}
\begin{Lemma}\label{lem:32}
For any $n\geq 2$, let $\beta_n$ be the  n-bonacci number. Then for any $N\geq 3$,  
$$F_{\beta_n, N-1}\subset E_{\beta_n, N}.$$
\end{Lemma}
\begin{proof}
 Since $\beta_n$ is the Pisot number satisfying  the  equation $\beta^n=\beta^{n-1}+\beta^{n-2}+\cdots +\beta+1,$  it follows that the quasi-greedy expansion of $1$ is $(1^{n-1}0)^{\infty}$. Hence the block $1^{n-1}$ can appear in the greedy expansions. 
 In other words, any expansion in base $\beta_n$  can be changed into the greedy expansions using the rule $10^{n}\sim 01^{n}$, i.e. the block $10^{n}$ can be replaced by $01^{n}$ without changing the value of the corresponding number. Given any point $x\notin E_{\beta_n, N}$, there exists an expansion of $x$ such that its coding, say $(a_n)$,  consists of a block $(0\cdots 0)$ with length $N$, i.e. there exists some $k_0$ such that 
 $a_{k_0+1}\cdots a_{k_0+N}=0\cdots 0$.  If $(a_n)$ is the greedy expansion of $x$, then 
clearly $x\notin F_{\beta_n, N-1}$. Assume $(a_n)$ is not the greedy expansion. We can
transform $(a_n)$ into the greedy expansion of $x$ by using the rule $10^{n}\sim 01^{n}$. Denote the acquired greedy expansion of $x$ by $(b_n)$.  Notice that the used transformation shrinks a block of zeros in the sequence  $(a_n)$ by at most one term. To be more precise, if $a_{k_0+1}\cdots a_{k_0+N}a_{k_0+N+1}a_{k_0+N+n}=\underbrace{0\cdots 0}_{N} 1^n$, then the corresponding block is $$b_{k_0+1}\cdots b_{k_0+N}b_{k_0+N+1}b_{k_0+N+n}=\underbrace{0\cdots 0 }_{N-1}10^n.$$
Thus, $x\notin F_{\beta_n, N-1}$.
\end{proof}
Next, we want to prove that
$$\lim\limits_{N\to \infty}\dim_{H}(F_{\beta_n, N})=1.$$
This result can be   obtained by the perturbation theory, it is essentially proved by Ferguson and Pollicott \cite[Theorem 1.2]{FP}.   
\begin{Lemma}\label{lem:31}
For any $1<\beta<2$, 
$\lim\limits_{N\to \infty}\dim_{H}(F_{\beta, N})=1$.
\end{Lemma}
Here we give a detailed proof of our desired limit,
\begin{Lemma}\label{lem:34}
$$\lim\limits_{N\to \infty}\dim_{H}(F_{\beta_n, N})=1.$$
\end{Lemma}
For simplicity, we assume $n=2$, for $n\geq3$ the proof is similar but the calculation is more complicited. 
We give an outline of the proof of this lemma. First, we give a Markov partition for $[0,(\beta-1)^{-1}]$ using the orbit of $1$.  Hence, we can define an adjacency matrix $S$ and construct an associated subshift of finite type $\Sigma$. Equivalently, we transform the original space $\{0,1\}^{\mathbb{N}}$ into a subshift of finite type. 
Next, we define a submatrix $S^{'}$ of $S$, and construct a graph-directed self-similar set with the open set condition \cite{MW}. Finally, we identify $F_{\beta_n, N}$ with a graph-directed self-similar set, and prove the desired result. 
Now we transform the symbolic space as follows.
 \begin{Lemma}\label{lem:35}
Let $\beta=\dfrac{1+\sqrt{5}}{2}$, and $x\in [0,(\beta-1)^{-1}]$. Then the greedy expansion of  $x$ has a coding which is from some subshift of finite type. 
 \end{Lemma}
\begin{proof}
Firstly, we give a Markov partition for the interval $[0,(\beta-1)^{-1}]$ as follows:
let $$a_1=0,a_i=\beta^{-N-2+i}(\beta-1)^{-1}, 2\leq i\leq n-1,a_{n}=\beta^{-1}=\beta^{-2}(\beta-1)^{-1},$$ $a_{n+1}=1, a_{n+2}=(\beta-1)^{-1}.$
Define
$$A_{1}=[0,\beta^{-N}(\beta-1)^{-1}],A_i=[\beta^{-N+i-2}(\beta-1)^{-1},\beta^{-N+i-1}(\beta-1)^{-1}], 2\leq i\leq N,$$ $A_{N+1}=[1,( \beta-1)^{-1}]$.
It is easy to check that 
$$T_0(A_1)=A_1\cup A_2, T_0(A_i)=A_{i+1}, 2\leq i\leq N-1,$$ 
and that 
$T_1(A_N)=\cup_{i=1}^{N-1}A_i,$ $T_1(A_{N+1})=A_N\cup A_{N+1}.$
Hence, we have the following adjacency matrix $S=(s_{ij})_{(N+1)\times (N+1)}$
\begin{equation*}
s_{ij}=\left\lbrace\begin{array}{cc}
                 1&  i=1, j=1,2\\

               1&  2\leq i\leq N-1, j=i+1\\
               1&  i= N, j=1,2,\cdots, N-1\\
               1&  i= N+1, j=N,N+1\\
               0&   else\\
                \end{array}\right.
\end{equation*}
In terms of $S$, we can construct a subshift of finite type $\Sigma$. 
For any $$(\alpha_{i})\in \{1,2,\cdots,N+1\}^\mathbb{N},$$
we call $\{A_{\alpha_{i}}\}_{i=1}^{\infty}$ an admissible path if there is some $T_k, $ $k=0$ or $1$, such that $$T_k(A_{\alpha_{i}})\supset A_{\alpha_{i+1}}$$ for any $i\geq 1.$ In terms of this definition, we have that 
$$\Sigma=\{(\alpha_{i})_{i=1}^{\infty}:\alpha_i\in\{1,2,\cdots,N+1\}, \{A_{\alpha_{i}}\}_{i=1}^{\infty} \mbox{ is an admissible path}\}.$$
\end{proof}
\setcounter{Remark}{5}
\begin{Remark}
Usually, we take the elements of Markov partition $A_i$ closed on the left and open on the right, i.e. $A_i=[a_i, a_{i+1})$.  Under our algorithm, one has a choice at the endpoints. For example the point $\beta^{-1}$ it is the right endpoint of $A_{N-3}$ and the left endpoint of $A_{N-2}$. For this point, we can implement $T_0$ on $A_k=[a_k, \beta^{-1}]$ or $T_1$ on $[\beta^{-1}, a_{k+2}]$. This adjustment is due to the proof of Lemma \ref{lem:39}. When we construct a graph-directed self-similar set, we need the closed interval, see the graph-directed construction in \cite{MW}.  This is the reason why we need some compromise here. 
Although
 our Markov partition is a little diffferent from the usual definition,  this adjustment does not affect our result.
\end{Remark}
By definition of  $E_{\beta, N}$, for any point $x\in E_{\beta, N}$,  all possible orbits of $x$ do not hit  the hole $A_{1}=[0,\beta^{-N}(\beta-1)^{-1}]$, which is the first element of the Markov partition. By Lemma \ref{lem:35}, $x$ also has a coding in the new symbolic space $\Sigma.$
For simplicity, we denote this coding of $x$ in $\Sigma$ by $\{\alpha_{i_n}\}_{n=1}^{\infty}$. Since $x\in E_{\beta, N}$,  the symbol $1$ cannot appear in the coding $\{\alpha_{i_n}\}_{n=1}^{\infty}$.  Motivated by this observation, we construct a new matrix as follows. We delete the first row and  first column of $S$, and keep the rest of the matrix.   Denote the new resulting matrix by $S^{'}$, and the associated subshift generated by $S^{'}$ is denoted by $\Sigma^{'}$.  $S^{'}$ can be represented by a directed graph $(V, E)$.  The vertex set consists of the underlying partition  $\{A_{i}\}_{i=2}^{k}$.
For two vertices, if one vertex  is one of components of  the  image of another vertex, then we can find a similitude, which is the inverse of the expanding map, between these two vertices. For instance, for the vertices $A_2$ and $A_3$, if $T_0(A_2)=A_3$, then we can label a directed edge, from the vetex $A_2$ to $A_3$, by a similitude $f(x)=T^{-1}_0(x)=\dfrac{x}{\beta}$. We denote all admissible labels between two vertices by $E$. Then by Mauldin and Williams' result \cite{MW}, we can construct a graph-directed self-similar set $K_N^{'}$ satisfying the open set condition, for the detailed construction,  see  \cite{MW, KarmaKan}. 
 Now we have the following lemma.
 \setcounter{Lemma}{6}
 \begin{Lemma}\label{lem:37}
Let $\beta=\dfrac{1+\sqrt{5}}{2}.$  $F_{\beta, N}=K_{N}^{'}$ except for a countable set, i.e. there exists a countable set $C_1$ such that 
$F_{\beta, N}\subset K_{N}^{'}\subset C_1\cup F_{\beta, N}$. 
\end{Lemma}
\begin{proof}
Evidently, $F_{\beta, N}\subset K_{N}^{'}.$
Take $x\in K_{N}^{'}$. Then by the definition of $K_{N}^{'}$, the greedy orbit of $x$ does not hit $[0,\beta^{-N}(\beta-1)^{-1})$. If the greedy orbit of $x$ does not hit the closed interval $[0,\beta^{-N}(\beta-1)^{-1}]$, 
then $x\in F_{\beta, N}$. 
If there exists some $(i_1i_2\cdots i_{n_0})$ such that 
$$T_{i_1i_2\cdots i_{n_0}}(x)=\beta^{-N}(\beta-1)^{-1},$$
then $$x\in \cup_{n=1}^{\infty}\cup_{(i_1\cdots i_n)\in \{0,1\}^n}f_{i_1\cdots i_n}(\beta^{-N}(\beta-1)^{-1}),$$ where $f_0(x)=\beta^{-1}x, f_1(x)=\beta^{-1}x+\beta^{-1}.$
Therefore,
$$K^{\prime}_{ N}\subset E_{\beta, N}\cup \cup_{n=1}^{\infty}\cup_{(i_1\cdots i_n)\in \{0,1\}^n}f_{i_1\cdots i_n}(\beta^{-N}(\beta-1)^{-1}).$$
\end{proof}
\begin{Lemma}\label{lem:39}
Let $\beta=\dfrac{1+\sqrt{5}}{2}.$  
Then $$\dim_{H}(F_{\beta, N})=\dfrac{\log \lambda_N}{\log \beta},$$ where $\lambda_N$ is the largest positive root of  the following equation 
$$x^{N-1}=\sum_{i=0}^{N-3}x^{i}. $$
Moreover,
$\lim\limits_{N\to \infty}\lambda_N=\dfrac{1+\sqrt{5}}{2}=\beta$. 
\end{Lemma}
\begin{proof}
By Lemma \ref{lem:37}, $\dim_{H}(F_{\beta, N})=\dim_{H}(K^{'}_N).$
$K^{'}_N$  is a graph-directed self-similar set with the open set condition, as such we can  explicitly calculate its Hausdorff dimension, namely, $\dim_{H}(F_{\beta, N})=\dfrac{\log \lambda_N}{\log \beta}$, where  $\lambda_N$ is indeed the spectral radius of $S^{'}$, for the detailed method, see \cite{MW}. The second statement is a simple exercise. We finish the proof of Lemma \ref{lem:34} for the case $n=2$. For $n\geq 3$, the proof is similar. 
\end{proof}
Similar result is available for the doubling map with hole \cite{GlendinningSidorov}. 
Let 
$D(x)=2x \mod 1$ be the doubling map defined on $[0,1)$. Given any $\epsilon>0$, 
set 
$$D_{\epsilon}=\{x\in [0,1): D^{n}(x)\notin [0, \epsilon] \mbox{ for any } n\geq 0\}.$$
Clearly $\lim\limits_{\epsilon\to 0}\dim_{H}(D_{\epsilon})$ exists. Hence, it  suffices to consider the following set 
$$D_{2^{-N}}=\{x\in [0,1): D^{n}(x)\notin [0, 2^{-N}] \mbox{ for any } n\geq 0\}$$
if we want to find  $\lim\limits_{\epsilon\to 0}\dim_{H}(D_{\epsilon})$. 
We have the following result. 
\setcounter{Example}{8}
\begin{Example}
$$\dim_{H}(D_{2^{-N}})=\dfrac{\log \gamma_N}{\log 2},$$
where $\gamma_N$ is the N-bonacci number satisfying the equation 
$$x^{N}=x^{N-1}+x^{N-2}+\cdots+x+1.$$
It is easy to see that $\lim\limits_{N\to \infty} \gamma_N=2$. Therefore, 
$$\lim\limits_{\epsilon\to 0}\dim_{H}(D_{\epsilon})=\lim\limits_{N\to \infty}\dim_{H}(D_{2^{-N}})=1.$$
\end{Example}
\begin{proof}[\bf{Proof of Theorem \ref{MainTheorem}}]
Let $\beta_n$ be a  $n$-bonacci number. 
By Lemmas \ref{lem:34} and  \ref{lem:37}, we have 
 $$E_{\beta_n, N}\subset F_{\beta_n, N} \subset K_{N}^{'} \subset  F_{\beta_n, N} \cup C_1.$$
 Finally,  by Lemma \ref{lem:34},
$$\lim_{N\to \infty}\dim_{H}(F_{\beta_n, N})=1.$$
Therefore, $$\dim_{H}(F_{\beta_n, N})= \dim_{H}(E_{\beta_n, N})\leq \dim_{H}(V_{\beta_n})\leq 1 ,$$
which implies that 
$\dim_{H}(V_{\beta_n})=1$. 
\end{proof}

It is easy to show that when $\beta$ is a Pisot number, then all the possible orbits of  $x\in \mathbb{Q}([\beta])\cap [0,(\beta-1)^{-1}]$ hit finitely many points only.  The following lemma is standard. However for the sake of convenience, we give the detailed proof. 
\setcounter{Lemma}{9}
\begin{Lemma}\label{lem:310}
Suppose $\beta$ is a  Pisot number and  $x\in \mathbb{Q}([\beta])\cap [0,(\beta-1)^{-1}]$, then the set 
$$\{\pi(K^{n}(\omega, x)):n\geq 0, \omega\in \Omega\}$$
is a  finite set.
\end{Lemma}
\begin{proof}
Let $M(X)=X^d-q_1X^{d-1}-\cdots-q_d$ be the minimal polynomial of $\beta$ with $q_i\in \mathbb{Z}.$ Since 
$\mathbb{Q}([\beta])$ is generated by $\{\beta^{-1},\cdots, \beta^{-d}\}$, there exist $a_1, a_2\cdots, a_d\in \mathbb{Z}$ and $b\in \mathbb{N}$ such that 
$$x=b^{-1}\sum_{i=1}^{d}a_i\beta^{-i}.$$
We assume that $b$ is as small as possible to ensure uniqueness. Let $\beta_1=\beta,$ and $\beta_2, \cdots, \beta_d$ the Galois conjugates of $\beta$, and set $B=(\beta_{j}^{i})_{1\leq i,j\leq d}$. Define for $n\geq 0$ and $\omega\in \Omega,$
$$r_n^{(1)}(\omega)=\beta^n\left(x-\sum_{k=1}^{n}b_k(\omega, x)\beta^{-k}\right)$$
and 
$$r_n^{(j)}(\omega)=\beta_{j}^{n}\left(b^{-1}\sum_{i=1}^{d}a_i\beta_{j}^{-i}-\sum_{k=1}^{n}b_k(\omega, x)\beta_{j}^{-k}\right)$$
for $j=2,3,\cdots, d.$ 
Consider the vector $R_n(\omega)=(r_{n}^{(1)}(\omega), \cdots, r_{n}^{(d)}(\omega)).$
We first show that the set $\{R_n(\omega):n\geq 0, \omega\in \Omega\}$ is uniformly bounded (in $n$ and $\omega$). First note that 
$r_{n}^{(1)}(\omega)=\pi(K^{n}(\omega, x))$, hence $|r_{n}^{(1)}(\omega)|\leq (\beta-1)^{-1}$ for any $n$ and any $\omega.$ Let $\eta=\max_{2\leq j\leq d}|\beta_j|$, then $\eta<1$. For $j=2,\cdots, d$
\begin{eqnarray*}
|r_{n}^{(j)}|&=& \left|\left(b^{-1}\sum_{i=1}^{d}a_i\beta_{j}^{n-i}-\sum_{k=1}^{n}b_k(\omega, x)\beta_{j}^{n-k}\right)\right|\\&\leq
&\left|\left(b^{-1}\sum_{i=1}^{d}|a_i|\eta^{n-i}\right)\right|+\left|\left( \sum_{k=1}^{n}b_k(\omega, x)\eta^{n-k}\right)\right|\\&\leq&
\dfrac{b^{-1}\max_{1\leq i\leq d}|a_i|+1}{1-\eta}
\end{eqnarray*}
Let $C=\max\left\{(\beta-1)^{-1},  \dfrac{b^{-1}\max_{1\leq i\leq d}|a_i|+1}{1-\eta}\right\}$, then $r_n^{(j)}<c$
for any $1\leq j\leq d, n\geq 0$ and $\omega\in \Omega.$ Thus the set $\{R_n(\omega):n\geq 0, \omega\in \Omega\}$ is uniformly bounded. Next we show that for each $\omega\in \Omega$ and $n\geq 0$, there exists $P(X)\in \mathbb{Z}_{n}(\omega)\in \mathbb{Z}^{d}$, then $\beta_2,\cdots, \beta_d$ are also roots of $P(X)$, it suffices to show that 
$$r_n^{(1)}=b^{-1}\left(\sum_{k=1}^{d}z_n^{k}(\omega)\beta^{-k}\right)$$
for $z_n^{k}\in \mathbb{Z}.$ The proof is done by contradiction. Let $n=1$ and note that $1=q_1\beta^{-1}+\cdots+q_d\beta^{-d}.$ Now
\begin{eqnarray*}
r_{1}^{(1)}(\omega)&=& \beta x- b_1(\omega, x)\\&=
&\beta b^{-1}\sum_{k=1}^{d}a_k\beta^{-k}-b_1(\omega, x) \sum_{k=1}^{d}q_k\beta^{-k}\\&=&
b^{-1}(\sum_{k=1}^{d-1}(a_1q_k-b_1(\omega, x)bq_k+a_{k+1})\beta^{-k}+(a_1-b_1(\omega, x)b)q_d\beta^{-d} )\\&=&b^{-1}\sum_{k=1}^{d}z_{1}^{(k)}(\omega)\beta^{-k}\end{eqnarray*}
with 
\begin{equation*}
z_{1}^{(k)}(\omega)=\left\lbrace\begin{array}{cc}
                 (a_1-b_1(\omega, x)b)q_k+a_{k+1}&  \mbox{ if } k\neq d\\

               (a_1-b_1(\omega, x)b)q_d&  \mbox{ if } k= d\\
                \end{array}\right.
\end{equation*}
Suppose now that 
$r_i^{(1)}=b^{-1}\sum_{k=1}^{d}z_{i}^{(k)}(\omega)\beta^{-k}$ for $z_{i}^{(k)}\in \mathbb{Z}.$ Since $r_{i}^{(1)}=\pi(K^{n}(\omega, x))$ for all $n\geq 0,$ we have 
\begin{eqnarray*}
r_{i+1}^{(1)}&=& \beta r_{i}^{(1)}-b_{i+1}(\omega, x)\\&=
&\beta b^{-1}\sum_{k=1}^{d}z^{(k)}_{i}(\omega)\beta^{-k}-b_{i+1}(\omega, x) \sum_{k=1}^{d}q_k\beta^{-k}\\&=&
b^{-1}(\sum_{k=1}^{d-1}(z^{(1)}_{i}q_k-b_{i+1}(\omega, x)bq_k+z^{(k+1)}_{i})\beta^{-k}+(z^{(1)}_{i}(\omega)-b_{i+1}(\omega, x)b)q_d\beta^{-d} )\\&=&b^{-1}\sum_{k=1}^{d}z_{i+1}^{(k+1)}(\omega)\beta^{-k}
\end{eqnarray*}
with 
\begin{equation*}
z_{i+1}^{(k)}(\omega)=\left\lbrace\begin{array}{cc}
                 (z_{i}^{(1)}(\omega)-b_{i+1}(\omega, x)b)q_k+z_{i}^{(k+1)}(\omega)&  \mbox{ if } k\neq d\\

               (z_{i}^{(1)}(\omega)-b_{i+1}(\omega, x)b)q_d&  \mbox{ if } k= d\\
                \end{array}\right.
\end{equation*}
Thus, $z_{i+1}^{(k)}(\omega)\in \mathbb{Z}.$ Setting $\mathbb{Z}(\omega)=(z_n^{(1)}, \cdots,z_n^{(d)}), $ we have $\mathbb{Z}(\omega)\in \mathbb{Z}^{d}$ and $R_n(\omega)=b^{-1}\mathbb{Z}(\omega)B$. Since $B$ is invertible, and $R_n(\omega)$ is uniformly bounded in $n$ and $\omega$, we have that $\mathbb{Z}(\omega)$ is uniformly bounded, and hence takes only finitely many values. It follows that $(R_n(\omega))$ takes only finitely many values. Therefore, the set 
$$\{\pi(K^{n}(\omega, x)):n\geq 0, \omega\in \Omega\}$$
 is finite. 
\end{proof}
\setcounter{Corollary}{10}
\begin{Corollary}\label{cor:36}
Let $\beta\in(1,2)$ be a Pisot number. For any  $a_{i_1}a_{i_2}\cdots a_{i_n}\in \{0,1\}^n$, the orbits of the endpoints of the interval    $f_{a_{i_1}a_{i_2}\cdots a_{i_n}}([0,(\beta-1)^{-1}])$ hit finite points.
\end{Corollary}
\begin{proof}
By symmetry, we only need to prove  that for the left endpoint 
$\sum_{j=1}^{n}a_{i_j}\beta^{-j}$, all of its orbits hit finite points. This is a directly consequence of Lemma \ref{lem:310}.
\end{proof}
\begin{proof}[\bf{Proof of Corollary \ref{cor:23} and Theorem \ref{thm:24}}]
By Lemma \ref{lem:310} and Corollary \ref{cor:36} and the main result of Mauldin and Williams \cite{MW}, we can calculate the Hausdorff dimension of $\dim_{H}(F_{\beta, N})$.
By Lemma \ref{lem:32} and Corollary \ref{cor:23}, we have the asymptotic result. 
For the second statement of Theorem \ref{thm:24}, we define 
$$D=\{10^{i_1}10^{i_2}10^{i_3}\cdots:n+1\leq i_k\leq N-1 \mbox{ and  there are infinitely many } i_k=N-1\}.$$
By Theorem \ref{Uniquecodings}, all the codings in $D$ are greedy in base $\beta_n$. Clearly, $D$ has uncountably many elements. Moreover,  the following inclusion holds,
$$p(D)=\left\{\sum_{j=1}^{\infty}a_j\beta^{-j}:(a_j)\in D\right\}\subset F_{\beta_n, N}.$$
Now we want to show that $p(D)\cap E_{\beta_n, N}=\emptyset$ and $\dim_{H}(F_{\beta_n, N}\setminus E_{\beta_n, N})>0.$
By the definition of $D$, for any $10^{i_1}10^{i_2}10^{i_3}\cdots$  there are infinitely many   $i_k=N-1$. Without loss of generality, we
assume that $i_1=N-1,$ i.e.  let $$ (a_k)=10^{N-1}10^{i_2}10^{i_3}\cdots.$$
Using the rule $10^n\sim 01^n$, we have 
$$x=(10^{N-1}10^{i_2}10^{i_3}\cdots)_{\beta}=(10^{N}1^n0^{i_2-n}10^{i_3}\cdots)_{\beta}\notin E_{\beta, N},$$
where $(b_k)_{\beta}=\sum_{k=1}^{\infty}b_k\beta^{-k}$. Hence, $p(D)\cap E_{\beta_n, N}=\emptyset$.
In order to prove $\dim_{H}(F_{\beta, N}\setminus E_{\beta, N})>0,$ it suffices to show that 
$\dim_{H}(p(D))>0.$
Here,  the set $(D,\sigma)$ is indeed a subset of some $S$-gap shift \cite{LM}, i.e. $D\subset D^{'}$, where 
$$D^{'}=\{10^{i_1}10^{i_2}10^{i_3}\cdots:n+1\leq i_k\leq N-1\}.$$
The entropy of $D^{'}$ can be calculated, 
i.e. $h(D^{'})=\log \lambda$, where $\lambda$ is the largest positive root of the equation 
$$1=\sum_{k\in\{n+1,\cdots,N-1\}}x^{-k-1} $$ 
Now we construct a subset of $p(D)$ as follows:
let $J$ be the self-similar set with the IFS $$\left\{g_1(x)=\dfrac{x}{\beta^{n+2}}+\dfrac{1}{\beta},g_2(x)=\dfrac{x}{\beta^{N}}+\dfrac{1}{\beta}\right\},$$
i.e. 
$$J=g_1(J)\cup g_2(J).$$
By the definitions of $p(D)$ and $J$, $J\subset p(D)$.
Let  $$E:=(\beta^{N-1}(\beta^N-1)^{-1}, \beta^{n+1}(\beta^{n+2}-1)^{-1}).$$
It is easy to check that 
$g_1(E)\cap g_2(E)=\emptyset,$ and $g_i(E)\subset E$. In other words, the IFS satisfies the open set condition \cite{Hutchinson}. Hence, 
$\dim_{H}(J)=s>0$,
where s is the unique solution of the equation $\beta^{(-n-2)s}+\beta^{-Ns}=1.$
Subsquently,  $$0<\dim_{H}(J)=s\leq\dim_{H}(p(D)).$$
For the last statement of Theorem \ref{thm:24}, it suffices to consider the set $p(D).$
\end{proof}
Now we prove Theorem \ref{thm:25}. We partition the proof into several lemmas. 
The following result is essentially proved in \cite{SKK}. For convenience, we give the detailed proof. 
\setcounter{Lemma}{11}
\begin{Lemma}\label{lem:315}
Given $1<\beta<2$ and let $N\geq 3$. If there exists some $(\eta_1\eta_2\cdots \eta_p)\in \{0,1\}^{p}$ such that $T_{\eta_1\eta_2\cdots \eta_p}(\beta^{-N}(\beta-1)^{-1})\in (0, \beta^{-N}(\beta-1)^{-1})$, then   $E_{\beta, N}$ is a graph-directed self-similar set.
\end{Lemma}
\begin{proof}
By assumption and the continuity of the $T_{j}$'s, there exists $\delta>0$ such that\[T_{\eta_{1}\ldots \eta_{p}}(\beta^{-N}(\beta-1)^{-1}, \beta^{-N}(\beta-1)^{-1}+\delta)\subset(0, \beta^{-N}(\beta-1)^{-1}).\] 
Set $H=[0, \beta^{-N}(\beta-1)^{-1}+\delta]$.
We partition $[0, (\beta-1)^{-1}]$ in terms of  the iterated function system \[f_j(x)=\dfrac{x+j}{\beta},\,j\in\{0,1\}.\]
For any $L$ we have $$[0, (\beta-1)^{-1}]=\bigcup_{(i_{1},\ldots,i_{L})\in\{1,\ldots,m\}^{L}}f_{i_{1}}\circ\cdots\circ f_{i_{L}}([0, (\beta-1)^{-1}]).$$ We assume  without loss of generality that $L$ is sufficiently large such that $$|f_{i_{1}}\circ\cdots\circ f_{i_{L}}([0, (\beta-1)^{-1}])|<\delta$$ for all $(i_{1},\ldots,i_{L})\in\{0,1\}^{L}$.  Correspondingly, we partition  the symbolic space $\{0,1\}^{\mathbb{N}}$ provided by the cylinders of length $L$. For every $(i_{1},\ldots, i_{L})\in\{0,1\}^{L}$ let $$C_{i_{1}\ldots i_{L}}=\Big\{(x_{n})\in\{0,1\}^{\mathbb{N}}:x_{n}=i_{n} \textrm{ for } 1\leq n\leq L\Big\}.$$ The set $\{C_{i_{1}\ldots {i_{L}}}\}_{(i_{1},\ldots ,i_{L})\in\{0,1\}^{L}}$ is a partition of $\{0,1\}^{\mathbb{N}},$ and  $f_{i_{1}}\circ\cdots\circ f_{i_{L}}([0, (\beta-1)^{-1}])=\pi(C_{i_{1}\ldots {i_{L}}})$.
Let \[\mathbb{F}=\Big\{(i_{1},\ldots,i_{L})\in\{1,\ldots,m\}^{L}:f_{i_{1}}\circ\cdots\circ f_{i_{L}}([0, (\beta-1)^{-1}])\cap [0, \beta^{-N}(\beta-1)^{-1}]\neq\emptyset\Big\}\]
and
\[\mathbb{F^{'}}=\bigcup_{(i_{1},\ldots,i_{L})\in \mathbb{F}}\pi(C_{i_{1}\ldots {i_{L}}}).\]
 By our assumptions on the size of our cylinders the following inclusions hold \[[0, \beta^{-N}(\beta-1)^{-1}]\subset \mathbb{F^{'}}\subset H.\]
Using these inclusions we can show that  $x\notin E_{\beta, N}$ if and only if there exists $(\theta_1,\ldots,\theta_{n_1})\in\{1,\ldots,m\}^{n_1}$ such that  $T_{\theta_1\ldots\theta_{n_1}}(x)\in \mathbb{F^{'}}$.   If $x\notin E_{\beta, N}$ then by the above observation, there exists $(\theta_1,\ldots,\theta_{n_1})\in\{1,\ldots,m\}^{n_1}$ such that  $T_{\theta_1\ldots\theta_{n_1}}(x)\in \mathbb{F^{'}}$. Therefore, $x$ has a coding containing a block from $\mathbb{F}$.  Conversely, if there exists $(\theta_1,\ldots,\theta_{n_1})\in\{1,\ldots,m\}^{n_1}$ such that  $T_{\theta_1\ldots\theta_{n_1}}(x)\in \mathbb{F^{'}}$, then  the condition 
\[T_{\eta_{1}\ldots \eta_{p}}(\beta^{-N}(\beta-1)^{-1}, \beta^{-N}(\beta-1)^{-1}+\delta)\subset(0, \beta^{-N}(\beta-1)^{-1})\] 
yields $x\notin E_{\beta, N}$. 
Taking $\mathbb{F}$ to be the set of forbidden words defining a subshift of finite type, we see that  $E_{\beta, N}$  is a graph-directed self-similar set, see \cite{KarmaKan, MW}.
\end{proof}
Schmeling \cite{Schmeling} proved the following result.
\begin{Lemma}\label{lem:316}
For almost every $\beta\in(1,2)$, the greedy orbits of $1$ and the lazy orbit of  $\bar{1}=(\beta-1)^{-1}-1$ are dense. 
\end{Lemma}
\begin{proof}[\bf{Proof of Theorem \ref{thm:25}}]
Theorem \ref{thm:25} follows immediately from Lemmas \ref{lem:315} and \ref{lem:316}.
\end{proof}
\section{Final remarks}
Similar results are available if we consider $\beta$-expansions with more than two digits. 
For some Pisot numbers, we may  implement similar ideas which are utilized in Lemmas \ref{lem:34} \ref{lem:32}. Finally we pose a problem.
\begin{Problem}
Does there exist $\delta>0$ such that for any $\beta\in(2-\delta,2)$, $\dim_{H}(V_{\beta})=1.$

\end{Problem}

\section*{Acknowledgements}
The second author was granted by the National Science Foundation of China no.11671147. The authors would like to thank  Nikita Sidorov for the discussion of Lemma \ref{lem:32} and other related problems.

\end{document}